%% file: main_0113.tex
\documentclass[12pt,reqno]{amsart}

\usepackage[a4paper,margin=1in]{geometry}
\usepackage{amsmath,amssymb,amsthm,mathtools}
\usepackage{bbm}
\usepackage{enumitem}
\usepackage{hyperref}
\usepackage[normalem]{ulem}

\usepackage{pdflscape}
\usepackage{subcaption}
\usepackage{caption}
\usepackage{tikz, float, tikzscale}
\usepackage{pgfplots}
\pgfplotsset{compat=1.18}

\usepackage[T1]{fontenc}
\usepackage{libertinus,libertinust1math}

\theoremstyle{definition}
\newtheorem*{defi}{Definition}
\theoremstyle{plain}
\newtheorem{thm}{Theorem}
\newtheorem*{thm*}{Theorem}
\newtheorem{prop}{Proposition}
\newtheorem{lemma}{Lemma}

\newtheorem{coro}{Corollary}

\title{Prescribed distinct-digit growth in countable alphabets}
\author{Ying Wai Lee}

\begin{document}

\begin{abstract}
    The number of distinct symbols appearing in digit expansions generated by full-branch affine countable iterated function systems is studied whose branch weights are regularly varying. The Hausdorff dimensions of the exceptional sets in which the distinct-digit count grows at a positive linear rate or at a prescribed sublinear rate are determined. The resulting dimension laws exhibit a sharp phase transition: imposing any positive linear rate forces the dimension to collapse to a value determined solely by the tail index, whereas under a broad class of sublinear growth rates, the exceptional sets retain full Hausdorff dimension.
\end{abstract}

\maketitle

\section{Introduction}

Distinct-value statistics form a fundamental class of observables at the interface of probability, number theory, dynamical systems, and fractal geometry. Let $\Xi\coloneqq(\xi_n)_{n\in\mathbb{N}}$ be a symbolic process taking values in a countable alphabet, identified here with $\mathbb{N}$. A natural question is: how many \emph{distinct} symbols have appeared up to time $n$? Define the \emph{distinct-digit count} $D_n$ by, for any $n\in\mathbb{N}$,
\begin{align*}
    D_n\coloneqq \#\{\xi_1,\ldots,\xi_n\}.
\end{align*}

This single quantity has several classical interpretations. In probability, it is the number of \emph{occupied boxes} in an infinite-urn occupancy scheme after a given number of draws; in number-theoretic expansions, it is a measure of \textit{digit diversity}; in dynamical systems, it is a non-additive observable derived from a symbolic coding; in fractal geometry, prescribing its growth naturally leads to exceptional sets whose size is quantified by Hausdorff dimension.

The present study focuses on affine full-branch countable iterated function systems on the unit interval $[0,1)$: the interval is partitioned into countably many subintervals, each branch maps its subinterval affinely onto the whole interval, and one digit is produced at each step. Under Lebesgue measure, the digit process is independent and identically distributed, so distinct-digit count is simultaneously identified with an infinite-urn occupancy count and with a natural dynamical statistic for the associated interval map with its coding cylinder metric. The classical L\"uroth expansion is a distinguished example within this framework.

Under regularly varying (or heavy-tailed) digit weights, the Hausdorff dimensions of sets are determined where distinct-digit count grows atypically, both at positive linear growth and a broad class of prescribed sublinear growth. The resulting dimension law exhibits a sharp phase transition: sublinear prescriptions typically remain full-dimensional, while any positive linear rate forces a universal dimension drop governed solely by the tail index.

\subsection{Continued fractions}

Continued fraction expansions provide a number-theoretic setting in which a digit process on a countable alphabet is generated by a well-understood dynamical system. Classical work has emphasised magnitude-related phenomena for partial quotients (such as extremes and growth rates), and more recently longest-run statistics have also been studied \cite{LeeLongestRunPartialQuotients2026,TanZhou2025UniformDiophantineRunLength,WangWu2011MaxRunLength}. An equally natural, but less systematically explored, statistic concerns \emph{diversity}: at any given cut-off, how many distinct integer values have occurred among the partial quotients. Each first appearance of a new partial quotient can be viewed as genuinely new symbolic information. Unlike the classical infinite-urn occupancy model, continued-fraction digits are not independent; nevertheless, under Gauss measure, their long-range statistics are sufficiently regular to support a refined Hausdorff-dimension theory for level sets of distinct-digit count.

For any $x\in[0,1)\setminus\mathbb{Q}$, there exists a unique sequence of positive integers $(a_n)_{n\in\mathbb{N}}$ such that $x$ admits the continued fraction expansion:
\begin{align*}
    x=[a_1,a_2,a_3,\ldots]\coloneqq\frac{1}{\displaystyle a_1+\frac{1}{\displaystyle a_2+\frac{1}{a_3+\cdots}}},
\end{align*}
where for any $n\in\mathbb{N}$, $a_n\coloneqq a_n(x)\in\mathbb{N}$ is referred to as $n$-th partial quotient of $x$. The sequence of range functions $(R_n)_{n\in\mathbb{N}}$ is defined by, for any $n\in\mathbb{N}$ and $x\in[0,1)\setminus\mathbb{Q}$, 
\begin{align*}
    R_n(x)\coloneqq \#\left\{a_1(x),a_2(x),\ldots,a_n(x)\right\},
\end{align*}
which counts the number of distinct values among the first $n$ partial quotients of $x$. For example, one obtains $R_{30}(\pi-3)=10$ by observing the first 30 partial quotients of $\pi-3$:
\begin{align*}
    \pi-3
    =[\dotuline{7},\dotuline{15},\dotuline{1},\dotuline{292},1,1,1,\dotuline{2},1,\dotuline{3},1,\dotuline{14},2,1,1,2,2,2,2,1,\dotuline{84},2,1,1,15,3,\dotuline{13},1,\dotuline{4},2,\ldots],
\end{align*}
where the underlined terms represent the first occurrence of each distinct integer. 

A classical consequence of the metric theory of continued fractions is that the sequence of partial quotients is unbounded for Lebesgue almost every irrational. Equivalently, for Lebesgue almost every $x\in[0,1)\setminus\mathbb{Q}$, $\lim_{n\to+\infty}R_n(x)=+\infty$. A natural quantitative question is therefore the typical growth rate of the range functions. Wu--Xie \cite{WuXie2017RangeRenewalCF} initiated a systematic study of the range-renewal structure, and obtained an explicit almost-sure growth law with an effective convergence rate.
\begin{thm*}[Wu--Xie {\cite[Remark~2]{WuXie2017RangeRenewalCF}}]
    For Lebesgue almost every $x\in[0,1)\setminus\mathbb{Q}$, as $n\to+\infty$,
    \begin{align}
    \label{eq: Wu--Xie (2.9)}
        \frac{R_n(x)}{\sqrt{n}}
        =\sqrt{\frac{\pi}{\log 2}}+O\!\left(n^{-\lambda}\right),
    \end{align}
    where the exponent is given by:
    \begin{align*}
        \lambda\coloneqq\frac{1}{6(1+2\log{2})}=0.06984\ldots.
    \end{align*}
\end{thm*}

The Hausdorff dimensions of exceptional sets associated with growth constraints on the range function were also studied by Wu--Xie \cite{WuXie2017RangeRenewalCF}. Define, for any $\beta\geq 0$ and $c>0$, the growth rate level set $E(\beta,c)$ by:
\begin{align*}
    E(\beta,c)
    \coloneqq \left\{x\in[0,1)\setminus\mathbb{Q}: \lim_{n\to+\infty}\frac{R_n(x)}{n^\beta}=c\right\}.
\end{align*}

\begin{thm*}[Wu--Xie {\cite[Theorem~2 \& Remark~3]{WuXie2017RangeRenewalCF}}]\label{thm:WX-dim}
\leavevmode
\begin{enumerate}
    \item
    For any $0<\beta<1$ and $c>0$,
    \begin{align*}
        \dim E(\beta,c)=1.
    \end{align*}
    \item
    For any $0<c\leq 1$,
    \begin{align*}
        \dim E(1,c)=\frac12.
    \end{align*}
    \item
    For any non-decreasing and unbounded $\psi:\mathbb{N}\to\mathbb{R}^+$, if $\lim_{n\to+\infty}\psi(n)/n=0$ then:
    \begin{align*}
        \dim{\left\{x\in[0,1)\setminus\mathbb{Q}: 
        \lim_{n\to+\infty}\frac{R_n(x)}{\psi(n)}=1\right\}}=1.
    \end{align*}
    \end{enumerate}
\end{thm*}

These results exhibit a sharp \emph{dimension dichotomy}: a broad class of sublinear prescriptions yield full Hausdorff dimension, whereas any positive linear distinctness rate forces a universal drop to $1/2$. This naturally raises the question of how robust such a phase transition is for other digit systems, and which structural parameters determine the value of the dimension drop.

\subsection{Affine full-branch countable iterated function systems}

More generally, digit sequences may be viewed as symbolic codings arising from countable iterated function systems. Attention is restricted to a canonical \emph{affine, full-branch} subclass on the unit interval, for which the induced digit process is independent and identically distributed under Lebesgue measure. In this independent setting, the model may be regarded as a convenient prototype for distinct-value statistics.

Let $(p_k)_{k\in\mathbb N}$ be a probability sequence, that is, $\sum_{k\in\mathbb N}p_k=1$ and for any $k\in\mathbb{N}$, $p_k>0$. Let $(I_k)_{k\in\mathbb N}$ be a partition of $[0,1)$ into pairwise disjoint half-open intervals. Suppose for any $k\in\mathbb{N}$,
\begin{align*}
    \operatorname{diam}{I_k}=p_k.
\end{align*}
Define the associated family of affine contractions $(\varphi_k)_{k\in\mathbb N}$ by, for any $k\in\mathbb{N}$ and $u\in[0,1)$,
\begin{align*}
    \varphi_k(u)\coloneqq \inf I_k + p_k\,u
\end{align*}
Then for any $k\in\mathbb{N}$, $\varphi_k([0,1))=I_k$ and the images $(\varphi_k([0,1)))_{k\in\mathbb N}$ are pairwise disjoint by the half-open convention. Define the full-branch map $T:[0,1)\to[0,1)$ by, for any $k\in\mathbb{N}$ and $x\in I_k$,
\begin{align*}
    T(x)\coloneqq\frac{x-\inf I_k}{p_k}
\end{align*}
Thus, for any $k\in\mathbb{N}$, $T|_{I_k}={\varphi_k}^{-1}$ is the branch inverse, and $T$ maps the interval $I_k$ affinely onto $[0,1)$. Define, for any $n\in\mathbb{N}$, the digit function $d_n:[0,1)\to\mathbb{N}$ by, for any $k\in\mathbb{N}$ and $x\in T^{-(n-1)}(I_k)$,
\begin{align*}
    d_n(x)\coloneqq k.
\end{align*}
Then for any $x\in[0,1)$, $(d_n(x))_{n\in\mathbb N}$ is the coding (symbolic itinerary) of $x$ with respect to $(\varphi_k)_{k\in\mathbb N}$; that is, for any $x\in[0,1)$ and $n\in\mathbb{N}$, 
\begin{align*}
    x=\left(\varphi_{d_1(x)}\circ\varphi_{d_2(x)}\circ\cdots\circ\varphi_{d_n(x)}\right)\left(T^{n}x\right).
\end{align*}

Define, for $n\in\mathbb{N}$, the distinct-digit count function $D_n:[0,1)\to\mathbb{N}$ by, for any $x\in[0,1)$, 
\begin{align*}
    D_n(x)\coloneqq \#{\left\{d_1(x),d_2(x),\ldots,d_n(x)\right\}}.
\end{align*}
A key simplification of the affine full-branch setting (in contrast with continued fractions) is that under Lebesgue measure the digit process $(d_n)_{n\in\mathbb N}$ is independent and identically distributed \cite{DajaniKraaikamp1996ApproximationLuroth,JagerDeVroedt1969LurothErgodic}. In particular, for any $n\in\mathbb{N}$ and $k\in\mathbb{N}$,
\begin{align*}
    \mathbb{P}(d_n=k)=\operatorname{diam}{I_k}=p_k.
\end{align*}
Consequently, the count of distinct digits is equal to the number of occupied boxes in an infinite-urn occupancy scheme, with box probabilities determined by the branch weights. Sharp asymptotic results from occupancy theory are thus applicable, and serve as probabilistic input for the dimension results stated later. Within this framework, the classical L\"uroth expansion is taken as a guiding case.

\subsection{Classical L\"uroth expansions}

As a classical countable-alphabet expansion arising from a deterministic full-branch map, L\"uroth expansions have been studied extensively. Their metric theory has been developed from several perspectives, including digit-frequency and multifractal analysis \cite{Barbour2009UnivariateApproximations,Feng2023DimensionTO}, growth rate exceptional sets for digit processes \cite{FengTanZhou2021ExactDimensions,LinLiLou2022LargestDigits,TanZhang2020RelativeGrowth}, and Diophantine approximation by L\"uroth convergents \cite{CaoWuZhang2013Efficiency,HuangKalle2024GeneralisedAlphaLuroth,Lee2025LurothDA,TanZhou2021ApproximationProperties}. These results illustrate that, beyond probabilistic heuristics, L\"uroth digits exhibit rich structure typical of continued-fraction-type expansions.

For any $x\in(0,1]$, there exists a unique sequence of positive integers $(d_n)_{n\in\mathbb{N}}$ such that $x$  admits the classical L\"uroth expansion \cite{Luroth1883Expansion}:
\begin{align*}
    x=\sum_{n\in\mathbb{N}}\frac{1}{d_n\prod_{j=1}^{n-1}d_j(d_j-1)}=\frac{1}{d_1}+\frac{1}{d_1(d_1-1)d_2}+\frac{1}{d_1(d_1-1)d_2(d_2-1)d_3}+\cdots,
\end{align*}
where for any $n\in\mathbb{N}$, $d_n\coloneqq d_n(x)\in\mathbb{N}\setminus\{1\}$ is referred to as the $n$-th digit of $x$. Equivalently, one introduces the canonical partition $(J_k)_{k\in\mathbb N}$ of $(0,1]$ by, for any $k\in\mathbb{N}$,
\begin{align*}
    J_k=\left(\frac{1}{k+1},\frac{1}{k}\right],
\end{align*}
and defines the L\"uroth map $T:(0,1]\to(0,1]$ as associated full-branch map by, for any $k\in\mathbb{N}$ and $x\in J_k$,
\begin{align*}
    T(x)
    =\frac{x-\inf J_k}{\operatorname{diam}{J_k}}
    =k(k+1)x-k
\end{align*}
In Figure~\ref{fig: both maps}, the L\"uroth map plays an affine full-branch role analogous to the Gauss map in continued fractions.
\input{gauss_luroth_maps}

Define, for any $n\in\mathbb{N}$, the digit function $d_n:(0,1]\to\mathbb{N}\setminus\{1\}$ by, for any $k\in\mathbb{N}$ and $x\in T^{-(n-1)}J_k$,
\begin{align*}
    d_n(x)\coloneqq k+1.
\end{align*}
$(d_n)_{n\in\mathbb{N}}$ coincides with the digits in the classical L\"uroth expansion. Since Lebesgue measure coincides with interval length and the digit process $(d_n)_{n\in\mathbb N}$ is independent and identically distributed under Lebesgue measure, for any $n\in\mathbb{N}$ and $d\in\mathbb{N}\setminus\{1\}$,
\begin{align*}
    \mathbb P(d_n=d)=\operatorname{diam}{J_{d-1}}=\frac{1}{d(d-1)}.
\end{align*}

If one instead realises the system on $[0,1)$ by a half-open version of the partition, the coding differs only on a countable set of partition endpoints. The set has Lebesgue measure $0$ and Hausdorff dimension $0$; hence, the modification has no effect on statements concerning Lebesgue measure or Hausdorff dimension for level sets defined in terms of the distinct-digit count.

\subsection{Classical infinite-urn occupancy theory}

Although the digit sequence is generated by a deterministic full-branch map, an independent and identically distributed description of the digit process is valid under the reference measure. The digit process is thereby represented as an infinite-urn occupancy scheme, with box probabilities given by the branch weights, and the count of distinct digits is equal to the number of boxes that have been occupied after a given number of draws.

Since the coding under Lebesgue measure has the same distribution as independent and identically distributed draws, statements that hold with probability one in classical occupancy theory are transferred into Lebesgue-almost-everywhere statements on the unit interval. Through this identification, sharp asymptotic results from occupancy theory are imported to describe the typical growth of distinct-digit count, and these estimates are used as the probabilistic backbone for the Hausdorff dimension results established later.

The typical growth of the distinct-digit count is well understood in broad heavy-tailed regimes. Suppose, for the moment, there exist $C>0$ and $\rho>1$ such that:
\begin{align}\label{eq:pk-power-intro}
    \lim_{k\to+\infty}\frac{p_k}{k^{-\rho}}=C.
\end{align}
Since $(D_n)_{n\in\mathbb{N}}$ coincides with the number of occupied boxes in the associated infinite-urn scheme, results from classical occupancy theory \cite{GnedinHansenPitman2007Occupancy,Karlin1967InfiniteUrnCLT} yield the $n^{1/\rho}$ law: for Lebesgue almost every $x\in[0,1)$,
\begin{align*}
    \lim_{n\to\infty}\frac{D_n(x)}{n^{1/\rho}}
    =\Gamma\!\left(1-\frac{1}{\rho}\right) C^{1/\rho},
\end{align*}
where $\Gamma$ is the gamma function. Moreover, law-of-the-iterated-logarithm type refinements provide essentially optimal almost-sure fluctuation scales around this leading term; for instance, one deduces from \cite{BuraczewskiIksanovKotelnikova2025LIL} that for Lebesgue almost every $x\in[0,1)$, as $n\to+\infty$,
\begin{align*}
    \frac{D_n(x)}{n^{1/\rho}}
    =\Gamma\!\left(1-\frac{1}{\rho}\right) C^{1/\rho}
    +O\!\left(n^{-1/(2\rho)}\sqrt{\log\log n}\right).
\end{align*}

In the classical L\"uroth setting, one deduces that for Lebesgue almost every $x\in(0,1]$, as $n\to+\infty$,
\begin{align*}
    \frac{D_n(x)}{\sqrt{n}}=\sqrt{\pi}+O\!\left(n^{-1/4}\sqrt{\log\log n}\right).
\end{align*}
These almost-everywhere laws identify the typical scale of distinct-digit count, but they do not address the geometric size of the sets where the growth of the distinct-digit count is \emph{atypical}. Motivated by the continued-fraction results of Wu--Xie \cite{WuXie2017RangeRenewalCF} and by the occupancy interpretation above, the Hausdorff dimensions of exceptional sets specified by asymptotic prescribed growth constraints are determined in the sequel.

\section{Main Results} 

Building on the typical growth discussed previously, attention now turns to exceptional sets defined by prescribed asymptotic growth of the distinct-digit count. The analysis is carried out under a regularly varying tail assumption, which is more general than the pure power-law condition~\eqref{eq:pk-power-intro}.
\begin{defi}\label{def:RV}
Let $(p_k)_{k\in\mathbb{N}}$ be a probability sequence and $\rho\geq 1$. $(p_k)_{k\in\mathbb{N}}$ is said to be \emph{regularly varying} with tail index $\rho$ if: there exists a slowly varying function $L:\mathbb{N}\to\mathbb{R}^+$ such that
\begin{align}\label{eq:rv}
    \lim_{k\to+\infty}\frac{p_k}{k^{-\rho}L(k)}=1.
\end{align}
\end{defi}

Theorem~\ref{thm:linear} concerns linear distinctness rates.
Define, for any $0<\theta\leq1$, the growth rate level set $E_\theta$ at linear scale $\theta$ by:
\begin{align*}
    E_\theta
    &\coloneqq \left\{x\in[0,1):\lim_{n\to+\infty}\frac{D_n(x)}{n}= \theta\right\}.
\end{align*}

\begin{thm}\label{thm:linear}
Let $(p_k)_{k\in\mathbb{N}}$ be a probability sequence and $\rho\geq 1$.
Suppose $(p_k)_{k\in\mathbb{N}}$ is regularly varying with tail index $\rho$.
Then for any $0<\theta\leq 1$,
\begin{align*}
    \dim{E_\theta}
    =\frac{1}{\rho}.
\end{align*}
\end{thm}
Thus, any positive linear rate forces a strict drop from full dimension (unless the tail index is 1), and the value is universal at the linear scale and determined solely by the tail index. In the classical L\"uroth setting, one obtains that for any $0<\theta\leq1$, $\dim{E_\theta}=1/2$, mirroring the continued-fraction phenomenon of Wu--Xie but in an independent and identically distributed digit setting.

The lower bound in Theorem~\ref{thm:linear} is proved by constructing a probability measure supported on a suitable subset and applying the mass distribution principle, with cylinder estimates derived from regular variation. The upper bound follows from a Hausdorff--Cantelli type covering argument combined with an appropriate change-of-measure (tilting) device for regularly varying tails.

Theorem~\ref{thm:sublinear} concerns a broad class of prescribed sublinear distinctness rates.
\begin{defi}\label{def:admissible}
Let $f:\mathbb{N}\to\mathbb{R}^+$ be a function. $f$ is said to be \emph{admissible} if:
\begin{itemize}
    \item $f$ is unbounded;
    \item $\lim_{n\to+\infty}f(n)\log{f(n)}/n=0$; and
    \item for any $n\in\mathbb{N}$, $0\leq f(n+1)-f(n)\leq1$.
\end{itemize}
\end{defi}

Define, for any admissible function $f$,
\begin{align*}
    E_f\coloneqq \left\{x\in[0,1): \lim_{n\to+\infty}\frac{D_n(x)}{f(n)}=1\right\}.
\end{align*}

\begin{thm}\label{thm:sublinear}
Let $(p_k)_{k\in\mathbb{N}}$ be a probability sequence and $\rho\geq 1$.
Suppose $(p_k)_{k\in\mathbb{N}}$ is regularly varying with tail index $\rho$.
Then for any admissible function $f$, 
\begin{align*}
    \dim{E_f}=1.
\end{align*}
\end{thm}
Thus, a broad class of prescribed sublinear distinctness rates retains full Hausdorff dimension, again paralleling the Wu--Xie phenomenon but in an independent and identically distributed digit setting. The proof of the full-dimensional phenomenon in Theorem~\ref{thm:sublinear} proceeds by an argument similar to that applied for the lower bound in Theorem~\ref{thm:linear}.

Corollary~\ref{cor:poly-phase} concerns sublinear polynomial prescriptions of the distinctness rate and the vanishing linear-rate set.
\begin{coro}\label{cor:poly-phase}
Let $(p_k)_{k\in\mathbb{N}}$ be a probability sequence and $\rho\geq 1$.
Suppose $(p_k)_{k\in\mathbb{N}}$ is regularly varying with tail index $\rho$.
Then:
\begin{itemize}
    \item for any $0<\beta<1$ and $c>0$, 
    \begin{align*}
        \dim{\left\{x\in[0,1):\lim_{n\to+\infty}\frac{D_n(x)}{n^\beta}=c\right\}}
        =1;
    \end{align*}
    \item and:
    \begin{align*}
        \dim{\left\{x\in[0,1):\lim_{n\to+\infty}\frac{D_n(x)}{n}=0\right\}}
        =1.
    \end{align*}
\end{itemize}
\end{coro}
Thus, the sets corresponding to sublinear polynomial prescriptions of the distinctness rate, as well as the vanishing linear-rate set, have full Hausdorff dimension. Corollary~\ref{cor:poly-phase} follows immediately by choosing admissible functions corresponding to the prescribed growth rates (with a finite modification if needed), applying Theorem~\ref{thm:sublinear} and set inclusions.

\section{Proof of Theorem~\ref{thm:linear}}

Let $0<\theta\leq1$. Define
\begin{align*}
    E_{\geq\theta,+}
    \coloneqq \left\{x\in[0,1):\limsup_{n\to+\infty}\frac{D_n(x)}{n}\geq \theta\right\}.
\end{align*}
Note that $E_\theta\subset E_{\geq\theta,+}$. To prove Theorem~\ref{thm:linear}, it suffices to prove the following double inequalities:
\begin{align*}
    \dim{E_{\geq\theta,+}}\leq \frac{1}{\rho} \leq\dim{E_\theta}.
\end{align*}

\subsection{Block concatenation construction}

The lower bound is obtained by constructing a subset via a block concatenation scheme, defining a probability measure supported on it, and then applying the mass distribution principle.

The following lemma is a standard consequence regarding dyadic blocks. Proofs may also be found in references on regular variation.
\begin{lemma}\label{lem:potter}
Let $(p_k)_{k\in\mathbb{N}}$ be a probability sequence and $\rho\geq 1$.
Suppose $(p_k)_{k\in\mathbb{N}}$ is regularly varying with tail index $\rho$.
Then there exists a slowly varying function $L:\mathbb{N}\to\mathbb{R}^+$ such that for any $\varepsilon>0$, there exist $k_{\varepsilon}\in\mathbb{N}$ and $C_\varepsilon\geq 1$ such that for any $k,m\in\mathbb{N}$, if $k_{\varepsilon}\leq k\leq m<2k$ then:
\begin{align*}
    \frac{p_m}{p_k}
    \geq\frac{1}{2^{\rho+\varepsilon}C_\varepsilon}, &&
    p_k\geq\frac{k^{-\rho}L(k)}{2}
    .
\end{align*}
\end{lemma}
\begin{proof}
    By the definition of regularly varying, there exists a slowly varying function $L:\mathbb{N}\to\mathbb{R}^+$ such that~\eqref{eq:rv} holds.
    Pick any $\varepsilon>0$.
    By the Potter bound of slowly varying function \cite[Theorem 1.5.6]{BinghamGoldieTeugels1987}, there exists $ k_{\varepsilon,1}\in\mathbb{N}$ and $D_\varepsilon>0$ such that for any $k,m\in\mathbb{N}$, if $k_{\varepsilon,1}\leq k\leq m<2k$ then:
    \begin{align*}
        \frac{L(m)}{L(k)}
        \geq\frac{1}{D_\varepsilon}\left(\frac{m}{k}\right)^{-\varepsilon}
        \geq\frac{1}{D_\varepsilon2^{\varepsilon}}.
    \end{align*}
    Define, for any $n\in\mathbb{N}$, $a_n\coloneqq p_n /(n^{-\rho}L(n))$.
    By regular variation~\eqref{eq:rv}, there exists $k_0\in\mathbb{N}$ such that for any $n\in\mathbb{N}$, if $n\geq k_0$ then $1/2\leq a_n\leq 3/2$. Hence, for any $k,m\in\mathbb{N}$, if $\max{\{k_{\varepsilon,1},k_0\}}\leq k\leq m<2k$ then:
    \begin{align*}
        \frac{p_m}{p_k}=\left(\frac{m}{k}\right)^{-\rho}\frac{L(m)}{L(k)}\frac{a_m}{a_k}\geq\frac{1}{3D_\varepsilon2^{\rho+\varepsilon}},
    \end{align*}
    and the desired lower bound follows.
\end{proof}

Define the increasing block length sequence $(L_j)_{j\in\mathbb{N}}$ by, for any $j\in\mathbb{N}$, $L_j\coloneqq 2^j$. 
Define, for any $j\in\mathbb{N}$, $m_j\coloneqq \lceil \theta L_j\rceil\leq L_j$ and $N_j\coloneqq 2^{j-1}\max{\{m_1,k_1\}}\in\mathbb{N}$, where $k_1$ is given in Lemma~\ref{lem:potter} by taking $\varepsilon=1$. 
Define the dyadic alphabets $(\mathcal{A}_j)_{j\in\mathbb{N}}$ by, for any $j\in\mathbb{N}$:
\begin{align*}
    \mathcal{A}_j\coloneqq \left\{N_j,N_j+1,\ldots,2N_j-1\right\}.
\end{align*}
Note that $(\mathcal{A}_j)_{j\in\mathbb{N}}$ is pairwise disjoint and for any $j\in\mathbb{N}$, $\#\mathcal{A}_j=N_j\geq m_j$.
Define, for any $j\in\mathbb{N}$, the target prefix distinctness profile $r_{j}:\{0,1,\ldots,L_j\}\to\{0,1,\ldots,L_j\}$ by, for any $t\in\{0,1,\ldots,L_j\}$,
\begin{align*}
    r_j(t)\coloneqq \left\lceil \theta t\right\rceil.
\end{align*}
and the associated new-digit times inside the block $I_j$:
\begin{align*}
    I_j\coloneqq\left\{t\in\{1,\ldots,L_j\}:r_j(t)=r_j(t-1)+1\right\}.
\end{align*}
Note that $\# I_j=r_j(L_j)=m_j$.
Define, for any $j\in\mathbb{N}$, the admissible blocks $\mathcal{B}_j\subset {\mathcal{A}_j}^{L_j}$ by requiring that the number of distinct symbols in each prefix is exactly $r_j(t)$:
\begin{align*}
    \mathcal{B}_j
    \coloneqq\bigcap_{t=1}^{L_j}\left\{(d_1,\ldots,d_{L_j})\in{\mathcal{A}_j}^{L_j}:\#\{d_1,\ldots,d_{t}\}=r_j(t)\right\}.
\end{align*}
That is, at times $t\in I_j$, the digit $d_t$ takes value that has not previously appeared in that block; at times $t\in \mathbb{N}\setminus I_j$, the digit $d_t$ repeats one of the already-seen values in that block.
Let $F_\theta$ be the set of all real numbers in the unit interval whose digit sequence is a concatenation of blocks from $(\mathcal{B}_j)_{j\in\mathbb{N}}$; that is,
\begin{align*}
    F_\theta
    \coloneqq
    \bigcap_{j\in\mathbb{N}}\left\{x\in[0,1):(d_{S_{j-1}+1}(x),\ldots,d_{S_j}(x))\in\mathcal{B}_j\right\},
\end{align*}
where $S_0\coloneqq 0$, and for any $J\in\mathbb{N}$, $S_{J}=\sum_{j=1}^J L_j$.

\begin{prop}\label{prop: thm 1 prop 1}
    For any $0<\theta\leq1$, $F_\theta\subset E_\theta$.
\end{prop}
\begin{proof}
Pick any $x\in F_\theta$. Since the alphabets $(\mathcal{A}_j)_{j\in\mathbb{N}}$ are disjoint, distinct digits contributed by different blocks do not overlap. Thus, for any $n\in\mathbb{N}$, there exists unique $J\in\mathbb{N}$ such that $S_{J-1}<n\leq S_J$ and:
\begin{align*}
    D_n(x)=\sum_{j=1}^{J-1}\#\left\{d_{S_{j-1}+1}(x),\ldots,d_{S_{j}}(x)\right\}+\#\left\{d_{S_{J-1}+1}(x),\ldots,d_{S_{J-1}+t}(x)\right\},
\end{align*}
where $t\coloneqq n-S_{J-1}\in\{1,\ldots,L_J\}$. By the construction of $(\mathcal{B}_j)_{j\in\mathbb{N}}$ and the definition of $(m_j)_{j\in\mathbb{N}}$,
\begin{align*}
    D_n(x)
    =\sum_{j=1}^{J-1}m_j+\left\lceil \theta t\right\rceil
    =\sum_{j=1}^{J-1}\left\lceil \theta L_j\right\rceil+\left\lceil \theta t\right\rceil.
\end{align*}
One obtains:
\begin{align*}
    \theta n=(S_{J-1}+t)\theta\leq D_n(x)< (S_{J-1}+t)\theta +J=\theta n+J.
\end{align*}
Note that $J<1+\log{n}/\log{2}$ and:
\begin{align*}
    0\leq\frac{D_n(x)}{n}-\theta\leq \frac{J}{n}<\frac{1+\log{n}/\log{2}}{n}.
\end{align*}
Thus, $\lim_{n\to+\infty}D_n(x)/n=\theta$.
\end{proof}

\begin{prop}\label{prop: thm 1 prop 2}
    For any $\delta>0$, there exists $j_{\delta}\in\mathbb{N}$ such that for any $j\in\mathbb{N}$, if $j\geq j_\delta$ then $\#\mathcal{B}_j\geq2$ and for any $(d_1,\ldots,d_{L_j})\in\mathcal{B}_j$,
    \begin{align}
    \label{eq:ratio_est}
        \frac{\log\#\mathcal{B}_j}{-\log \operatorname{diam}{C(d_1,\ldots,d_{L_j})}}
        \geq
        \frac{1-\delta}{\rho+\delta},
    \end{align}
    where for any $n\in\mathbb{N}$ and $(d_1,\ldots,d_{n})\in\mathbb{N}^n$,
    \begin{align*}
        C(d_1,\ldots,d_{n})
        =\bigcap_{i=1}^{n}\left\{x\in[0,1):d_i(x)=d_i\right\}.
    \end{align*}
\end{prop}
\begin{proof}
By Lemma~\ref{lem:potter} and the definition of $(N_j)_{j\in\mathbb{N}}$, there exist a slowly varying function $L:\mathbb{N}\to\mathbb{R}^+$ and $C_1\geq 1$ such that for any $j\in\mathbb{N}$ and $k\in \mathcal{A}_j\subset [N_j,2N_j)$,
\begin{align*}
    p_k\geq \frac{p_{N_j}}{2^{\rho+1}C_1},
\end{align*}
and
\begin{align*}
    p_{N_j}
    \geq \frac{{N_j}^{-\rho}L(N_j)}{2}.
\end{align*}
By combining the two inequalities above, one obtains that for any $j\in\mathbb{N}$ and $k\in \mathcal{A}_j$,
\begin{align*}
    p_k\geq \frac{{N_j}^{-\rho}L(N_j)}{2^{\rho+2}C_1};
\end{align*}
and, for any $(d_1,\ldots,d_{L_j})\in\mathcal{B}_j$ and $t\in\{1,\ldots,L_j\}$, one obtains $d_t\in \mathcal{A}_j$ and 
\begin{align*}
    \operatorname{diam} C(d_1,\ldots,d_{L_j})
    =\prod_{t=1}^{L_j}p_{d_t}
    \geq 
    \left(\frac{N_j^{-\rho}L(N_j)}{2^{\rho+2}C_1}\right)^{L_j}.
\end{align*}
By taking $K\coloneqq(\rho+2)\log2+\log{C_1}$, one obtains:
\begin{align}
\label{eq:-logdiam_upper}
    -\log \operatorname{diam} C(d_1,\ldots,d_{L_j})
    &\leq
    \rho L_j\log N_j - L_j\log L(N_j) + L_j\,K.
\end{align}

Note that, for any $j\in\mathbb{N}$, 
\begin{align}
\label{eq: count of B}
    \#\mathcal{B}_j =
    \frac{N_j!}{(N_j-m_j)!}\prod_{t\in \{1,\ldots,L_j\}\setminus{I_j}}r_j(t-1).
\end{align}
By standard bounds on factorials, one obtains:
\begin{align*}
    \frac{N_j!}{(N_j-m_j)!}
    \geq \frac{(N_j/e)^{N_j}}{(N_j-m_j)^{\,N_j-m_j}}
    = \left(\frac{N_j}{e}\right)^{m_j}\left(\frac{N_j}{N_j-m_j}\right)^{N_j-m_j}
    \geq \left(\frac{N_j}{e}\right)^{m_j}.
\end{align*}
Thus, for any $j\in\mathbb{N}$,
\begin{align}\label{eq:falling_factorial_lb}
    \log{\frac{N_j!}{(N_j-m_j)!}}\geq m_j\log N_j - m_j.
\end{align}
One obtains that for any $j\in\mathbb{N}$, 
\begin{align*}
    \prod_{t\in I_j\cap{\{2,\ldots,L_j\}}} r_j(t-1)=(m_j-1)!,
\end{align*}
and
\begin{align*}
    \prod_{t\in\{2,\ldots,L_j\}}r_j(t-1)\geq \prod_{u=1}^{L_j-1}\theta u=\theta^{L_j-1}(L_j-1)!.
\end{align*}
By combining both together, one obtains that for any $j\in\mathbb{N}$, 
\begin{align*}
    \prod_{t\in\{1,\ldots,L_j\}\setminus{I_j}} r_j(t-1)
    \geq \theta^{L_j-1}\frac{(L_j-1)!}{(m_j-1)!}
    \geq\theta^{L_j-1}{m_j}^{L_j-m_j},
\end{align*}
and
\begin{align}
\label{eq:repeat_product_lb}
    \log{\prod_{t\in\{1,\ldots,L_j\}\setminus{I_j}} r_j(t-1)}
    \geq (L_j-1)\log{\theta}+(L_j-m_j)\log{m_j}.
\end{align}
By combining~\eqref{eq: count of B},~\eqref{eq:falling_factorial_lb}, and~\eqref{eq:repeat_product_lb}, one obtains:
\begin{align*}
    \log{\#\mathcal{B}_j}
    &\geq
    m_j\log N_j - m_j  + (L_j-m_j)\log{m_j}+(L_j-1)\log{\theta} \\
    &= L_j\log{N_j}-(L_j-m_j)\log{\frac{N_j}{m_j}}-m_j+(L_j-1)\log{\theta}
    .
\end{align*}
By the definitions of $(m_j)_{j\in\mathbb{N}}$ and $(N_j)_{j\in\mathbb{N}}$, for any $j\in\mathbb{N}$, one obtains $1\leq{N_j}/{m_j}\leq N_1/\theta$ and:
\begin{align}
\label{eq: log Bjlower}
    \log{\#\mathcal{B}_j}\geq L_j\log{N_j}-c_\theta \,L_j,
\end{align}
where $c_\theta \coloneqq \log{N_1}-2\log{\theta}+1>0$.

Therefore, by combining~\eqref{eq:-logdiam_upper} and~\eqref{eq: log Bjlower}, for any $j\in\mathbb{N}$,
\begin{align*}
    \frac{\log\#\mathcal{B}_j}{-\log \operatorname{diam} C(d_1,\ldots,d_{L_j})}
    \geq
    \frac{\log{N_j} - c_\theta }{\rho\log{N_j} - \log{L(N_j)} +K}.
\end{align*}
Pick any $\delta>0$. Since $L$ is slowly varying and $(N_j)_{j\in\mathbb{N}}$ is unbounded, there exists $j_{\delta}\in\mathbb{N}$ such that for any $j\in\mathbb{N}$, if $j\geq j_{\delta}$ then $\#\mathcal{B}_j\geq2$ and:
\begin{align*}
    \frac{\max{\{2\left|\log L(N_j)\right|,2K,c_\theta\}}}{\log{N_j}}
    \leq\delta,
\end{align*}
in particular,
\begin{align*}
    \frac{\log\#\mathcal{B}_j}{-\log \operatorname{diam} C(d_1,\ldots,d_{L_j})}
    \geq
    \frac{(1-\delta)\log N_j}{(\rho+\delta)\log N_j}
    =
    \frac{1-\delta}{\rho+\delta}.
\end{align*}
\end{proof}

Define a probability measure $\mu$ supported on $F_\theta$ by assigning equal mass to each block cylinder: for $J\in\mathbb{N}$ and $(b_1,\ldots,b_J)\in\prod_{j=1}^J\mathcal{B}_j$, 
\begin{align*}
    \mu{\left(C(b_1,\ldots,b_J)\right)}
    \coloneqq\prod_{j=1}^J \frac{1}{\#\mathcal{B}_j}.
\end{align*}
\begin{prop}\label{prop: thm 1 prop 2.5}
    For any $0<\theta\leq1$ and $\delta>0$, there exist $H_{\theta,\delta}>0$ and $r_{\theta,\delta}>0$ such that for any interval $I\subset[0,1)$, if $\operatorname{diam}{I}<r_{\theta,\delta}$ then:
    \begin{align*}
        \mu{(I)}
        \leq H_{\theta,\delta}\,(\operatorname{diam}{I})^{(1-\delta)/(\rho+\delta)}.
    \end{align*}
\end{prop}
\begin{proof}
Pick any $\delta>0$. One obtains from~\eqref{eq:ratio_est} that for any $J\in\mathbb{N}$ and $(d_1,\ldots,d_{S_J})\in\prod_{j=1}^J\mathcal{B}_j$, 
\begin{align}
\label{eq: 2}
    \mu{\left(C{(d_1,\ldots,d_{S_J}})\right)}
    \leq A_{\theta,\delta}\left(\operatorname{diam}{C{(d_1,\ldots,d_{S_J}})}\right)^{(1-\delta)/(\rho+\delta)},
\end{align}
where $A_{\theta,\delta}\coloneqq\prod_{j=1}^{j_\delta-1}\sup_{b_j\in\mathcal{B}_j}(\operatorname{diam}{C{(d_1,\ldots,d_{S_j}})})^{-(1-\delta)/(\rho+\delta)}/\#\mathcal{B}_j<+\infty$.

Pick any interval $I\subset[0,1)$. Define the stopping time $\tau_I:I\to\mathbb{N}$ by, for any $x\in I$,
\begin{align*}
    \tau_I(x)
    \coloneqq
    \min{\left\{J\in\mathbb{N}:C{(d_1(x),\ldots,d_{S_J}(x))}\subset I\right\}}.
\end{align*}
Define a sequence of covers of pairwise disjoint maximal block cylinders $(\mathcal{C}_J(I))_{J\in\mathbb{N}}$ by, for any $J\in\mathbb{N}$,
\begin{align*}
    \mathcal{C}_J(I)\coloneqq\{P(J,x):x\in I\cap F_{\theta}\text{ and }\tau_I(x)=J\},
\end{align*}
where for any $x\in F_{\theta}$ and $J\in\mathbb{N}$,
\begin{align*}
    P(J,x)\coloneqq C{\left(d_1(x),\ldots,d_{S_J}(x)\right)}
\end{align*}
Thus, $\bigcup_{J\in\mathbb{N}}\mathcal{C}_J(I)$ is a cover of $I\cap F_{\theta}$. By the disjointness, one obtains:
\begin{align}
\label{eq: 3}
    \mu{(I)}=\mu{(I\cap F_\theta)}=\sum_{J\in\mathbb{N}}\sum_{C\in\mathcal{C}_J(I)}\mu(C).
\end{align}
Let $a(I,1)$ and $a(I,2)$ be the endpoints of $I$. By the maximality, for any $J\in\mathbb{N}$ and $C\in\mathcal{C}_J(I)$, $C$ is contained in one of the two boundary cylinders $P(J-1,a(I,1))$ or $P(J-1,a(I,2))$, and:
\begin{align}
\label{eq: 4}
    \sum_{C\in\mathcal{C}_J(I)}\mu(C)\leq \sum_{u\in\{1,2\}}\mu{\left(P(J-1,a(I,u))\right)}.
\end{align}
Let $r_{\theta,\delta}$ be a lower bound of the diameters of block cylinders in depth $j_\delta$ given by:
\begin{align*}
    r_{\theta,\delta}
    \coloneqq
    \prod_{j=1}^{j_\delta}\left(\min_{k\in\mathcal{A}_j}p_k\right)^{L_j}>0.
\end{align*}

Pick any interval $I\subset[0,1)$. Suppose $\operatorname{diam}{I}<r_{\theta,\delta}$. Then for any $x\in I\cap F_{\theta}$, one obtains $\tau_I(x)\geq j_\delta$. Define, for any $u\in\{1,2\}$, 
\begin{align*}
    J_u\coloneqq\min{\{J\in\mathbb{N}\cap[j_\delta,+\infty):\operatorname{diam}{P(J,a(I,u))}\leq \operatorname{diam}{I}\}}.
\end{align*}
By Proposition~\ref{prop: thm 1 prop 2}, for any $J\in\mathbb{N}$, if $J\geq j_\delta$ then $\#\mathcal{B}_{J+1}\geq2$ and for any $u\in\{1,2\}$,
\begin{align*}
    \mu{\left(P(J+1,a(I,u))\right)}
    \leq\frac{1}{\#\mathcal{B}_{J+1}}{\mu{\left(P(J,a(I,u))\right)}}
    \leq\frac{1}{2}\mu{\left(P(J,a(I,u))\right)}
    .
\end{align*}
Thus, one obtains that for any $u\in\{1,2\}$,
\begin{align}
\label{eq: 13}
    \sum_{J=J_u}^{+\infty}\mu{(P(J,a(I,u)))}
    \leq \mu{(P{(J_u,a({I,u})})}\sum_{m=0}^{+\infty}2^{-m}
    =2\mu{(P{(J_u,a({I,u})}))}.
\end{align}
By combining~\eqref{eq: 3}, ~\eqref{eq: 4} and~\eqref{eq: 13}, one obtains:
\begin{align*}
    \mu{(I)}
    \leq\sum_{J=j_\delta}^{+\infty}\sum_{u\in\{1,2\}}\mu{(P{(J-1,a({I,u})}))}
    \leq2\sum_{u\in\{1,2\}}\mu{(P{(J_u,a({I,u})}))}.
\end{align*}
By~\eqref{eq: 2} and the definitions of $J_1$ and $J_2$, for any $u\in\{1,2\}$,
\begin{align*}
    \mu{\left(P{(J_u,a({I,u})})\right)}
    \leq A_{\theta,\delta}\,(\operatorname{diam}{P{(J_u,a({I,u})})})^{(1-\delta)/(\rho+\delta)}
    \leq A_{\theta,\delta}\,(\operatorname{diam}{I})^{(1-\delta)/(\rho+\delta)}.
\end{align*}
Therefore, 
\begin{align*}
    \mu{(I)}
    \leq 4A_{\theta,\delta}\,(\operatorname{diam}{I})^{(1-\delta)/(\rho+\delta)}.
\end{align*}
\end{proof}

\begin{prop}\label{prop: thm 1 prop 3}
    For any $0<\theta\leq1$, $\dim{F_\theta}\geq 1/\rho$.
\end{prop}
\begin{proof}
By applying Proposition~\ref{prop: thm 1 prop 2.5} and \cite[Mass distribution principle~4.2]{Falconer2003FractalGeometry}, one obtains that for any $\delta>0$,
\begin{align*}
    \dim{F_\theta}\geq\frac{1-\delta}{\rho+\delta};
\end{align*}
hence the required lower bound follows.
\end{proof}

\subsection{Hausdorff--Cantelli and tilting}

It remains to prove the upper bound $\dim E_{\geq \theta,+}\leq 1/\rho$. For $\rho=1$, the desired bound is immediate as $E_{\geq \theta,+}\subset[0,1)$. Thus, one may assume $\rho>1$ and apply a Hausdorff--Cantelli argument combined with tilting.

The following lemma is a standard consequence in combinatorics.
\begin{lemma}\label{lem:distinctforceslarge}
Let $n\in\mathbb{N}$ and $(x_1,\ldots,x_n)\in\mathbb{N}^n$. Suppose $\#\{x_1,\ldots,x_n\}\geq m$. Then,
\begin{align*}
    \#\left\{1\leq i\leq n: x_i\geq \left\lceil\frac{m}{2}\right\rceil \right\}
    \geq \left\lceil\frac{m}{2}\right\rceil.
\end{align*}
\end{lemma}
\begin{proof}
Let $r\coloneqq \lceil m/2\rceil$ and $S\coloneqq\{x_1,\ldots,x_n\}$. Since
\begin{align*}
    \#\left(S\cap\{1,2,\ldots,r-1\}\right)\leq r-1,
\end{align*}
one obtains:
\begin{align*}
    \#\left\{1\leq i\leq n: x_i\geq r \right\}
    \geq\#\left(S\cap\{r,r+1,\ldots\}\right)
    \geq\#S-(r-1)
    \geq m-(r-1)
    =\left\lfloor\frac{m}{2}\right\rfloor+1.
\end{align*}
\end{proof}

The following lemma is a standard consequences regarding tails. Proofs may also be found in references on regular variation.
\begin{lemma}\label{lem:tailpk}
Let $(p_k)_{k\in\mathbb{N}}$ be a probability sequence and $\rho>1$.
Suppose $(p_k)_{k\in\mathbb{N}}$ is regularly varying with tail index $\rho$.
Then there exists a slowly varying function $L:\mathbb{N}\to\mathbb{R}^+$ such that for any $s\in(1/\rho,1)$, there exists $C_s>0$ and $M_{s}\in\mathbb{N}$ such that for any $M\in\mathbb{N}$, if $M\geq M_s$ then:
\begin{align*}
    \sum_{k=M}^{+\infty} {p_k}^{s} \leq C_s M^{1-\rho s}L(M)^{s}.
\end{align*}
\end{lemma}
\begin{proof}
    Pick any $s\in(1/\rho,1)$. By regular variation~\eqref{eq:rv}, there exists a slowly varying function $L:\mathbb{N}\to\mathbb{R}^+$ and $k_0\in\mathbb{N}$ such that for any $k\in\mathbb{N}$, if $k\geq k_0$ then $p_k\leq 2k^{-\rho}L(k)$. By taking $\varepsilon\coloneqq(\rho-1/s)/2>0$ and the Potter bound for slowly varying function \cite[Theorem 1.5.6]{BinghamGoldieTeugels1987}, there exist $M_0\in\mathbb{N}$ and $D_\varepsilon>0$ such that for any $k,M\in\mathbb{N}$, if $M_0\leq M\leq k$ then:
    \begin{align*}
        \frac{L(k)}{L(M)}\leq D_{\varepsilon}\left(\frac{k}{M}\right)^{\varepsilon}.
    \end{align*}
    Define $\beta\coloneqq(\rho-\varepsilon)s>1$. For any $M\in\mathbb{N}$, if $M\geq\max{\left\{k_0,M_0,2\right\}}$ then:
    \begin{align*}
        \sum_{k=M}^{+\infty}{p_k}^s
        \leq 2^{s}{D_\varepsilon}^s M^{-\varepsilon s} L(M)^s\sum_{k=M}^{+\infty}{k}^{-(\rho-\varepsilon)s}
        \leq\frac{2^{s+1}{D_\varepsilon}^s}{\beta-1}M^{1-\rho s} L(M)^s.
    \end{align*}
\end{proof}

\begin{prop}\label{prop: thm 1 prop 4}
    For any $0<\theta\leq 1$, $\dim{E_{\geq\theta,+}}\leq 1/\rho$.
\end{prop}
\begin{proof}
Define, for any $n\in\mathbb{N}$,
\begin{align*}
    E_{\geq\theta/2,n} \coloneqq \left\{x\in[0,1): \frac{D_n(x)}{n}\geq\frac{\theta}{2}\right\}.
\end{align*}
Notice that
\begin{align}
    \label{eq: prop 3 subset}
    E_{\geq\theta,+}\subset\limsup_{n\to+\infty}E_{\geq\theta/2,n}
\end{align}
Define $(\mathcal{W}_n(\theta))_{n\in\mathbb{N}}$ by, for any $n\in\mathbb{N}$,
\begin{align*}
    \mathcal{W}_n(\theta)
    \coloneqq\left\{\omega=(\omega_1,\ldots,\omega_n)\in\mathbb{N}^n:\frac{\#\{\omega_1,\ldots,\omega_n\}}{n}\geq\frac{\theta}{2}\right\}.
\end{align*}
Note that for any $n\in\mathbb{N}$,
\begin{align*}
    E_{\geq \theta/2,n}
    =\bigcup_{\omega\in\mathcal{W}_n(\theta)} C(\omega),
\end{align*}
and for any $N\in\mathbb{N}$,
\begin{align*}
    \limsup_{n\to\infty}E_{\geq \theta/2,n}
    \subset
    \mathcal{C}_N\coloneqq\bigcup_{n=N}^\infty \bigcup_{\omega\in\mathcal{W}_n(\theta)} C(\omega).
\end{align*}
Let $p_*\coloneqq\sup_{k\in\mathbb{N}}p_k\in(0,1)$. For any $n,N\in\mathbb{N}$ and $w\in\mathcal{W}_n(\theta)$, if $n\geq{N}$ then $\operatorname{diam}C(\omega)\leq {p_*}^n\leq {p_*}^N$ and $\mathcal{C}_N$ is a ${p_*}^N$-cover of $\limsup_{n\to\infty}E_{\geq \theta/2,n}$.
Thus, for any $s\in(1/\rho,1)$,
\begin{align*}
    \mathcal{H}^s_{{p_*}^N}{\left(\limsup_{n\to\infty}E_{\geq \theta/2,n}\right)}
    \leq
    \sum_{n=N}^\infty\sum_{\omega\in\mathcal W_n(\theta)}(\operatorname{diam}{C(\omega)})^s
    =
    \sum_{n=N}^\infty S_n(s,\theta),
\end{align*}
where for any $n\in\mathbb{N}$, $S_n(s,\theta)$ is the level-$n$ cylinder sum: 
\begin{align*}
    S_n(s,\theta)
    \coloneqq \sum_{\omega\in\mathcal{W}_n(\theta)} (\operatorname{diam}{C(\omega)})^s
    =\sum_{(\omega_1,\ldots,\omega_n)\in\mathcal{W}_n(\theta)}
    \prod_{i=1}^n {p_{\omega_i}}^{s}.
\end{align*}
By the set inclusion~\eqref{eq: prop 3 subset}, it remains to prove that $\sum_{n\in\mathbb{N}} S_n(s,\theta)<+\infty$, as:
\begin{align}
    \label{eq:HsEn}
    \mathcal{H}^s\left(\limsup_{n\to+\infty}E_{\geq\theta/2,n}\right)
    \leq \lim_{N\to+\infty}\sum_{n=N}^{\infty}S_n(s,\theta).
\end{align}

Define $Z_s\coloneqq \sum_{k\in\mathbb{N}} {p_k}^{s}<+\infty$. Let $(X_i)_{i\in\mathbb{{N}}}$ be a sequence independent and identically distributed random variables with marginal probabilities: for any $k\in\mathbb{N}$, 
\begin{align*}
    \mathbb{P}(X_1=k)=q_k\coloneqq \frac{{p_k}^{s}}{Z_s}.
\end{align*}
For any $n\in\mathbb{N}$ and $\omega\coloneqq(\omega_1,\ldots,\omega_n)\in\mathcal{W}_n(\theta)$, one obtains $\mathbb P((X_1,\ldots,X_n)=\omega)=\prod_{i=1}^n q_{\omega_i}$ and
\begin{align}
\label{eq:tilt}
    S_n(s,\theta)
    ={Z_s}^n\sum_{\omega\in\mathcal{W}_n(\theta)}\prod_{i=1}^nq_{\omega_i}
    ={Z_s}^n\,\mathbb{P}\left(\frac{\#\{X_1,\ldots,X_n\}}{n}\geq \frac{\theta}{2}\right).
\end{align}

Pick any $n\in\mathbb{N}$. Suppose $n\geq4M_s/\theta$, as given in Lemma~\ref{lem:tailpk}. By applying Lemma~\ref{lem:distinctforceslarge} with $m=m_n\coloneqq\lceil \theta n/2\rceil$ and $r_n\coloneqq \lceil m_n/2\rceil = \lceil \theta n/4\rceil$, one obtains the event inclusion:
\begin{align*}
    \left\{\frac{\#\{X_1,\ldots,X_n\}}{n}\geq \frac{\theta}{2}\right\}
    \subset \left\{\#\left\{1\leq i\leq n: X_i\geq r_n\right\}\geq r_n\right\},
\end{align*}
The right-hand side is a binomial event $\mathrm{Binomial}(n,q_{\geq r_n})$, where $q_{\geq {r_n}}\coloneqq\sum_{k={r_n}}^{+\infty}q_k$. By a standard upper bound of the binomial coefficients and Lemma~\ref{lem:tailpk}, one obtains:
\begin{align*}
    \mathbb{P}\left(\frac{\#\{X_1,\ldots,X_n\}}{n}\geq \frac{\theta}{2} \right)
    &\leq\left(\frac{enq_{\geq r_n}}{r_n}\right)^{r_n} \\
    &\leq\left(en\frac{C_s}{Z_s}{r_n}^{-\rho s}{L(r_n)}^s\right)^{r_n} \\
    &\leq\left(A_{\theta,s}n^{1-\rho s}{L(r_n)}^s\right)^{r_n}
    ,
\end{align*}
where $A_{\theta,s}\coloneqq (\theta/4)^{-\rho s}C_se/Z_s>0$. Since $L$ is slow varying, for $\varepsilon\coloneqq (\rho-1/s)/2>0$, there exists $B_s>0$ such that for any $k\in\mathbb{N}$,
\begin{align*}
    L(k)\leq B_s\,k^\varepsilon.
\end{align*}
Since $\theta n/4\leq r_n\leq n$, one obtains:
\begin{align*}
    \mathbb{P}\left(\frac{\#\{X_1,\ldots,X_n\}}{n}\geq \frac{\theta}{2} \right)
    &\leq \left(A_{\theta,s}{B_s}^sn^{-(\rho s-1)/2}\right)^{r_n} \\
    &\leq\exp{\left(-\frac{(\rho s-1)\theta}{8}n\log{n}+n\log{A_{\theta,s,1}}\right)}
    ,
\end{align*}
where $A_{\theta,s,1}\coloneqq \max\{1,A_{\theta,s}{B_s}^s\}$. Suppose:
\begin{align*}
    n\geq\exp{\frac{\log{Z_s}+\log{A_{\theta,s,1}}}{(\rho s-1)\theta/16}}.
\end{align*}
By~\eqref{eq:tilt},
\begin{align*}
    S_n(s,\theta)
    &\leq \exp{\left(n\log{Z_s}-\frac{(\rho s-1)\theta}{8}n\log{n}+n\log{A_{\theta,s,1}}\right)} \\
    &\leq \exp{\left(-\frac{(\rho s-1)\theta}{16}n\log{n}\right)},
\end{align*}
one obtains $\sum_{n\in\mathbb{N}} S_n(s,\theta)<+\infty$.

By~\eqref{eq:HsEn} and Hausdorff--Cantelli, 
\begin{align*}
    \mathcal{H}^s\left(\limsup_{n\to+\infty}E_{\geq\theta/2,n}\right)=0.
\end{align*}
Thus, one obtains from the definition of Hausdorff dimension that:
\begin{align*}
    \dim\left(\limsup_{n\to+\infty}E_{\geq\theta/2,n}\right)\leq s.
\end{align*}
Since $s\in(1/\rho,1)$ is arbitrary, the desired upper bound is obtained:
\begin{align*}
    \dim{E_{\geq\theta,+}}\leq \dim\left(\limsup_{n\to+\infty}E_{\geq\theta/2,n}\right)\leq \frac{1}{\rho}.
\end{align*}
\end{proof}

\begin{proof}[Proof of Theorem~\ref{thm:linear}]
    By combining Propositions~\ref{prop: thm 1 prop 1},~\ref{prop: thm 1 prop 3}, and~\ref{prop: thm 1 prop 4}, the proof of Theorem~\ref{thm:linear} is completed.
\end{proof}

\section{Proof of Theorem~\ref{thm:sublinear}}

Let $f:\mathbb{N}\to\mathbb{R}^+$ be an admissible function. The upper bound $\dim{E_f}\leq 1$ is trivially true; hence, it remains to prove the lower bound $\dim{E_f}\geq 1$. Without loss of generality, as $E_f=E_{\lfloor f\rfloor}$, $f$ is integer-valued. Without loss of generality, after permuting the partition intervals and relabelling the digits accordingly if necessary, the probability sequence $(p_k)_{k\in\mathbb{N}}$ is non-increasing. 

Define, $s_1\coloneqq0$ and for any $K\in\mathbb{N}\setminus\{1\}$, $s_K\in(0,1)$ to be the unique solution of
\begin{align}\label{eq:sM}
    \sum_{k=1}^K {p_k}^{s_K}=1.
\end{align}
\begin{prop}
    For any $0<t<1$, there exists a sequence $(K_n)_{n\in\mathbb{N}}$ of positive integers such that all of the following are satisfied
    \begin{enumerate}
        \item $(K_n)_{n\in\mathbb{N}}$ is non-decreasing and unbounded;
        \item
        for any $n\in\mathbb{N}$,
        \begin{align}\label{eq:sjgt}
            s_{K_n}\geq \frac{1+t}{2};
        \end{align}
        \item 
        there exists $n_t\in\mathbb{N}$ such that for any $n\in\mathbb{N}$, if $n\geq n_t$ then:
        \begin{align}\label{eq: K_n leq f(n)}
            K_n\leq \sqrt{f(n)}.
        \end{align}
    \end{enumerate}
\end{prop}
\begin{proof}
    Pick any $0<t<1$. 
    Since the partial sums $\sum_{k=1}^K p_k$ increase to 1 as $K\to+\infty$,  $(s_K)_{K\in\mathbb{N}}$ also increases to 1. There exists $K_t^*\in\mathbb{N}$ such that for any $K\in\mathbb{N}$, if $K\geq K_t^*$ then~\eqref{eq:sjgt} is satisfied. Define, for any $n\in\mathbb{N}$,
    \begin{align*}
        K_n\coloneqq\max{\left\{K_t^*,\left\lfloor\sqrt{f(n)}\right\rfloor\right\}}.
    \end{align*}
    Since $f$ is non-decreasing and $\lim_{n\to+\infty} f(n)=+\infty$, the sequence $(K_n)_{n\in\mathbb{N}}$ is non-decreasing and unbounded. Define 
    \begin{align*}
        n_t\coloneqq\min{\left\{n\in\mathbb{N}:K_t^*\leq \sqrt{f(n)}\right\}}.
    \end{align*}
    For any $n\in\mathbb{N}$, if $n\geq n_t$ then~\eqref{eq: K_n leq f(n)} is satisfied.
\end{proof}

Define the new-digit times set $\mathcal{T}$ by:
\begin{align*}
    \mathcal{T} \coloneqq \left\{n\in\mathbb{N}: f(n)=f(n-1)+1\right\},
\quad
\end{align*}
with the convention $f(0)\coloneqq0$.
Define the forced-digit pools $(\mathcal{P}_n)_{n\in\mathbb{N}}$ by, for any $n\in\mathbb{N}$,
\begin{align*}
    \mathcal{P}_n \coloneqq
    \begin{cases}
        \{K_n+f(n)\}, & n\in\mathcal{T}, \\
        \emptyset,    & n\in\mathbb{N}\setminus\mathcal{T}.
    \end{cases}
\end{align*}
Note that $(\mathcal{P}_n)_{n\in\mathbb{N}}$ is pairwise disjoint.
By~\eqref{eq: K_n leq f(n)}, one obtains that for any $n\in\mathbb{N}$, if $n\geq n_t$ then:
\begin{align}\label{eq:forcedsize}
    \max{\mathcal{P}_n}\leq2f(n),
\end{align}
with the convention $\max{\emptyset}=0$. Define the free alphabets $(\mathcal{A}_n)_{n\in\mathbb{N}}$ by, for any $n\in\mathbb{N}$,
\begin{align*}
    \mathcal{A}_n\coloneqq \{1,2,\ldots,K_n\}.
\end{align*}

Define the class of admissible sequences $(\mathcal{B}_n)_{n\in\mathbb{N}}$ by, for any $n\in\mathbb{N}$,
\begin{align*}
    \mathcal{B}_n \coloneqq
    \begin{cases}
        \mathcal{P}_n, & n\in\mathcal{T}, \\
        \mathcal{A}_n, & n\in\mathbb{N}\setminus\mathcal{T}.
    \end{cases}
\end{align*}
Let $F_{f,t}$ be the set of all real numbers in the unit interval whose digit sequence is a concatenation of digits from $(\mathcal{B}_n)_{n\in\mathbb{N}}$; that is,
\begin{align*}
    F_{f,t}
    \coloneqq
    \bigcap_{n\in\mathbb{N}}\left\{x\in[0,1):d_{n}(x)\in\mathcal{B}_n\right\},
\end{align*}

\begin{prop}\label{prop: F_t is a subset of E_f}
    For any $0<t<1$ and admissible function $f$, $F_{f,t}\subset E_f$.
\end{prop}
\begin{proof}
    Pick any $x\in F_{f,t}$ and $n\in\mathbb{N}$. One has the disjoint union:
    \begin{align*}
        \{1,\ldots,n\}
        =\left(\mathcal{T}\cap\{1,\ldots,n\}\right)\cup\left(\{1,\ldots,n\}\setminus\mathcal{T}\right).
    \end{align*}

    Pick any $j\in\{1,2,\ldots,n\}\cap\mathcal{T}$. By the construction of $F_{f,t}$,  the $j$-th digit of $x$ is forced to be:
    \begin{align*}
        d_j(x)=b_j\coloneqq K_j+f(j)\in \mathcal{P}_j.
    \end{align*}
    The value taken is genuinely new. Since $(K_n)_{n\in\mathbb{N}}$ is non-decreasing, for any $i\in\{1,2,\ldots,j-1\}\setminus\mathcal{T}$, $b_j>K_j\geq K_i\geq d_i(x)$.
    Since $(\mathcal{P}_n)_{n\in\mathcal{T}}$ is pairwise disjoint, one obtains that for any $i\in\{1,2,\ldots,j-1\}\cap\mathcal{T}$, $b_j\ne b_i$.
    Thus,
    \begin{align*}
        D_n(x)
        \geq
        \#\left(\mathcal{T}\cap\{1,\ldots,n\}\right)
        =\sum_{i=1}^n\left(f(i)-f(i-1)\right)
        =f(n).
    \end{align*}

    On the other hand, for any $i\in\{1,2,\ldots,n\}\setminus\mathcal{T}$, then the $i$-th digit of $x$ is free to choose from:
    \begin{align*}
        d_i(x)\in\mathcal{A}_i=\{1,2,\ldots,K_i\}.
    \end{align*}
    Since $(K_n)_{n\in\mathbb{N}}$ is non-decreasing, one obtains:
    \begin{align*}
        \bigcup_{i=1}^n \mathcal{A}_i=\mathcal{A}_n=\{1,\ldots,K_n\}.
    \end{align*}
    Thus, the first $n$-th digits take values from $\{b_i:i\in\{1,\ldots,n\}\cap\mathcal{T}\}\cup\mathcal{A}_n$ and:
    \begin{align*}
        D_n(x)
        \leq \#\{b_i:i\in\{1,\ldots,n\}\cap\mathcal{T}\}+\#{\mathcal{A}_n}
        =
        \#\left(\mathcal{T}\cap\{1,\ldots,n\}\right)+\#{\mathcal{A}_n} 
        =f(n)+K_n.
    \end{align*}

    By combining the last inequalities of both previous paragraphs, for any $n\in\mathbb{N}$,
    \begin{align*}
        f(n)\leq D_n(x) \leq f(n)+K_n.
    \end{align*}
    By~\eqref{eq: K_n leq f(n)}, one obtains that for any $n\in\mathbb{N}$, if $n\geq n_t$ then: 
    \begin{align*}
        1\leq\frac{D_n(x)}{f(n)}\leq 1+\frac{K_n}{f(n)}\leq1+\frac{1}{\sqrt{f(n)}}.
    \end{align*}
    Since $f$ is unbounded, $\lim_{n\to\infty} D_n(x)/f(n)=1$ follows and $x\in E_f$.
\end{proof}

Define the coding space $\Omega$ by:
\begin{align*}
    \Omega\coloneqq\prod_{n\in\mathbb{N}}\mathcal{B}_n,
\end{align*}
and the coding map $\Phi:\Omega\to[0,1)$ by, for any $(d_n)_{n\in\mathbb{N}}\in\Omega$,
\begin{align*}
    \Phi\left((d_n)_{n\in\mathbb{N}}\right)\coloneqq\bigcap_{n\in\mathbb{N}}C_n(d_1,d_2,\ldots,d_n).
\end{align*}
Define, for any $n\in\mathbb{N}$, $s_n\coloneqq s_{K_n}$ be the unique solution of~\eqref{eq:sM} with $K=K_n$. Define, for any $n\in\mathbb{N}$, a probability measure $\nu_{t,n}$ on $\mathcal{B}_n$ by:
\begin{align*}
    \nu_{t,n}\coloneqq
    \begin{cases}
    \displaystyle \sum_{k=1}^{K_n}{p_k}^{s_n}\delta_k, & n\in\mathbb{N}\setminus{\mathcal{T} },\\
    \delta_{b_n}, & n\in \mathcal{T} ,
    \end{cases}
\end{align*}
where $\delta$ denotes the Dirac probability measure, and $b_n\coloneqq K_n+f(n)$ denotes the unique forced digit in $\mathcal B_n$. 
Define $\nu_t\coloneqq\bigotimes_{n\in\mathbb{N}}\nu_{t,n}$ and a Borel probability measure $\mu_t$ supported on $F_{f,t}$ by the push-forward:
\begin{align*}
    \mu_t\coloneqq\Phi_*\nu_t.
\end{align*}

\begin{prop}\label{prop: length-cylinder}
    For any $0<t<1$, there exist $A_t>0$ and $c_t>0$ such that for any $n\in\mathbb{N}$,
    \begin{align*}
        \sup_{x\in F_{f,t}}\frac{\mu_t(C_n(x))}{\left({\operatorname{diam}{C_n(x)}}\right)^t}
        \leq {A_t}{e^{-c_tn}},
    \end{align*}
    where for any $x\in[0,1)$ and $n\in\mathbb{N}$, $C_n(x)$ is that rank-$n$ cylinder containing $x$; that is:
    \begin{align*}
        C_n(x)\coloneqq C\left(d_1(x),d_2(x),\ldots,d_n(x)\right)
    \end{align*}
\end{prop}
\begin{proof}
    Define, for any $n\in\mathbb{N}$, $U_n\coloneqq \{1\leq i\leq n: i\in\mathbb{N}\setminus \mathcal{T} \}$ to be the set of free positions up to $n$, and $V_n\coloneqq \{1\leq i\leq n: i\in \mathcal{T} \}$ to be the set of the forced positions up to $n$. Pick any $x\in F_{f,t}$. 
    By the construction of $\mu_t$, one obtains for any $n\in\mathbb{N}$,
    \begin{align*}
        \mu_t(C_n(x))=\prod_{i\in U_n} {p_{d_i(x)}}^{\,s_{i}}, &&
        \operatorname{diam}{C_n(x)}
        =\left(\prod_{i\in U_n} p_{d_i(x)}\right) \left(\prod_{i\in V_n} p_{d_i(x)}\right);
    \end{align*}
    hence,
    \begin{align}\label{eq:ratio}
        \frac{\mu_t(C_n(x))}{\left(\operatorname{diam}{C_n(x)}\right)^t}
        =\left(\prod_{i\in U_n} {p_{d_i(x)}}^{s_i-t}\right)
        \left(\prod_{i\in V_n} {p_{d_i(x)}}^{-t}\right).
    \end{align}

    By~\eqref{eq:sjgt}, for any $i\in\mathbb{N}$, $s_i-t\geq (1-t)/2$. Since $(p_k)_{k\in\mathbb{N}}$ is non-increasing, for any $n\in\mathbb{N}$, $i\in U_n$ and $x\in F_{f,t}$,
    \begin{align*}
        {p_{d_i(x)}}^{s_i-t}\leq {p_1}^{s_i-t}\leq {p_1}^{(1-t)/2}.
    \end{align*}
    By multiplying over $i\in U_n$, one obtains:
    \begin{align*}
        \prod_{i\in U_n} {p_{d_i(x)}}^{s_i-t}
        \leq \prod_{i\in U_n} {p_1}^{(1-t)/2}
        = {p_1}^{\,(1-t)\#U_n/2}
        = \exp{\left(\frac{(1-t)\log p_1}{2}\#U_n\right)}.
    \end{align*}
    Note that for any $n\in\mathbb{N}$, $\#U_n=n-\#V_n=n-f(n)$ and: 
    \begin{align}\label{eq:free_factor_bound}
        \prod_{i\in U_n} {p_{d_i(x)}}^{s_i-t}
        \leq \exp{\left(\frac{(1-t)\log p_1}{2}(n-f(n))\right)}.
    \end{align}

    Note that for any $n\in\mathbb{N}$ and $i\in V_n$, $d_i(x)\in\mathcal{P}$ is a forced digit. By~\eqref{eq:forcedsize} and $f$ being non-decreasing, for any $n\geq n_t$ and $i\in V_n$,
    \begin{align*}
        d_i(x)\leq 2f(i)\leq 2f(n).
    \end{align*}
    Since $(p_k)_{k\in\mathbb{N}}$ is non-increasing, one obtains:
    \begin{align*}
        -\log p_{d_i(x)} \leq -\log p_{2f(n)}.
    \end{align*}
    By summing over $i\in V_n$ and $\#V_n=f(n)$,
    \begin{align}\label{eq:forced_sum_reduce}
        -\sum_{i\in V_n} \log p_{d_i(x)}
        \leq-\#V_n \log p_{2f(n)}
        = -f(n)\log{p_{2f(n)}},
    \end{align}
    Pick any $\alpha>0$. By regular variation~\eqref{eq:rv} and the Potter bound for the slowly varying function \cite[Theorem 1.5.6]{BinghamGoldieTeugels1987}, there exists $K_\alpha\in\mathbb{N}$ such that for any $k\geq K_\alpha$,
    \begin{align*}
        p_k \geq \frac12 k^{-\rho}L(k)\geq \frac12 k^{-(\rho+\alpha)};
    \end{align*}
    hence, for any $k\geq K_\alpha$,
    \begin{align}\label{eq:logpk_upper}
        -\log p_k \leq (\rho+\alpha)\log k + \log 2.
    \end{align}
    Since $f$ is non-decreasing and unbounded, there exists $N_\alpha\in\mathbb{N}$ such that for any $n\in\mathbb{N}$, if $n\geq N_\alpha$ then $2f(n)\geq \max{\{K_\alpha,2\}}$. By applying~\eqref{eq:logpk_upper} with $k=2f(n)$, one obtains that for any $n\geq N_\alpha$,
    \begin{align*}
        -\log{p_{2f(n)}}
        \leq (\rho+\alpha)\log{(2f(n))}+\log 2
        \leq 2(\rho+\alpha)\log{f(n)}+\log 2.
    \end{align*}
    From~\eqref{eq:forced_sum_reduce}, one obtains that for any $n\in\mathbb{N}$, if $n\geq N_\alpha$ then:
    \begin{align*}
        - \sum_{i\in V_n}\log p_{d_i(x)}
        \leq f(n)\left(2(\rho+\alpha)\log{f(n)} + \log{2}\right).
    \end{align*}
    By the admissibility of $f$, there exists $n_\alpha\in\mathbb{N}$ such that for any $n\in\mathbb{N}$, if $n\geq n_\alpha$ then:
    \begin{align*}
        f(n)\left(2(\rho+\alpha)\log{f(n)} + \log{2}\right)
        \leq -\frac{(1-t)\log p_1}{3t}n.
    \end{align*}
    By combining the two inequalities above, one obtains that for any $n\in\mathbb{N}$, if $n\geq \max{\{N_\alpha,n_\alpha\}}$ then:
    \begin{align*}
        -\sum_{i\in V_n} \log p_{d_i(x)} \leq -\frac{(1-t)\log p_1}{3t}n;
    \end{align*}
    therefore,
    \begin{align}\label{eq:forced_factor_bound}
        \prod_{i\in V_n} {p_{d_i(x)}}^{-t}
        = \exp{\left(-t\sum_{i\in V_n}\log p_{d_i(x)}\right)}
        \leq \exp\!\left(-\frac{(1-t)\log p_1}{3}\,n\right).
    \end{align}
    
    By combining~\eqref{eq:ratio},~\eqref{eq:free_factor_bound}, and~\eqref{eq:forced_factor_bound}, for any $n\in\mathbb{N}$, if $n\geq \max{\{N_\alpha,n_\alpha\}}$ then:
    \begin{align*}
    \frac{\mu_t(C_n(x))}{(\operatorname{diam}C_n(x))^t}
    &\leq
    \exp{\left(\frac{(1-t)\log p_1}{2}(n-f(n))-\frac{(1-t)\log p_1}{3}n\right)}
    \\
    &=
    \exp{\left((1-t)\log p_1\left(\frac{n}{6}-\frac{f(n)}{2}\right)\right)},
    \end{align*}
    By the admissibility of $f$, there exists $n_1\in\mathbb{N}$ such that for any $n\in\mathbb{N}$, if $n\geq n_1$ then $f(n)\leq n/4$. Therefore, for any $n\in\mathbb{N}$ and $x\in F_{f,t}$, if $n\geq n^*\coloneqq\max{\{N_\alpha,n_\alpha,n_1\}}$ then:
    \begin{align*}
    \frac{\mu_t(C_n(x))}{(\operatorname{diam}C_n(x))^t}
    \leq
    \exp{\left(\frac{(1-t)\log p_1}{24}\,n\right)}.
    \end{align*}
    The proof is completed by taking:
    \begin{align*}
        c_t\coloneqq-\frac{(1-t)\log p_1}{24}>0, &&
        A_t\coloneqq\max{\left\{1,\max_{1\leq n\leq n^*}{\sup_{x\in F_{f,t}}{\frac{e^{c_tn}\mu_t(C_n(x))}{(\operatorname{diam}C_n(x))^t}}}\right\}}.
    \end{align*}
\end{proof}

\begin{proof}[Proof of Theorem~\ref{thm:sublinear}]
Since $E_f\subset[0,1)$, $\dim E_f\leq 1$; hence, it remains to prove $\dim E_f\geq 1$.

Pick any $0<t<1$. Define, for any interval $I\subset[0,1)$, the stopping time $\tau_I:I\to\mathbb{N}$ by, for any $x\in I$:
\begin{align*}
    \tau_I(x)\coloneqq\min{\{n\in\mathbb{N}:C_n(x)\subset I\}};
\end{align*}
and a sequence of covers of pairwise disjoint maximal cylinders $(\mathcal{C}_n(I))_{n\in\mathbb{N}}$ by, for any $n\in\mathbb{N}$,
\begin{align*}
    \mathcal{C}_n(I)
    \coloneqq\left\{C_{\tau_I(x)}(x)\in[0,1):x\in I\cap F_{f,t}\text{ and }\tau_I(x)=n\right\}.
\end{align*}
Note that $\bigcup_{n\in\mathbb{N}}\mathcal{C}_n(I)$  is a cover of $I\cap F_{f,t}$. By Proposition~\ref{prop: length-cylinder}, for any interval $I\subset[0,1)$,
\begin{align*}
    \mu_t(I)
    \leq \sum_{n\in\mathbb{N}}\sum_{C\in\mathcal{C}_n(I)}\mu_t(C)
    \leq A_t\,(\operatorname{diam}{I})^t\sum_{n\in\mathbb{N}}\frac{\#\mathcal{C}_n(I)}{e^{c_tn}}.
\end{align*}

By maximality, for any $n\in\mathbb{N}$ and $C\in \mathcal{C}_n(I)$, $C$ can only arise as children at time $n$ of the two rank $n-1$ boundary cylinders determined by the endpoints of $I$. For any $n\in\mathbb{N}$, the $n$-digit alphabet in the construction has size most $K_n+1$; hence $\#\mathcal{C}_n(I)\leq 2(K_n+1)$ and
\begin{align*}
    \mu_t(I)\leq 2A_t\,(\operatorname{diam}{I})^t\sum_{n\in\mathbb{N}}\frac{K_n+1}{e^{c_tn}}.
\end{align*}
By~\eqref{eq: K_n leq f(n)} and the admissibility of $f$,
\begin{align*}
    \sum_{n\in\mathbb{N}}\frac{K_n+1}{e^{c_tn}}
    \leq \sum_{n=1}^{n_t-1}\frac{K_n+1}{e^{c_tn}}+\sum_{n=n_t}^{+\infty}\frac{\sqrt{f(n)}+1}{e^{c_tn}}
    \leq \sum_{n=1}^{n_t-1}\frac{K_n+1}{e^{c_tn}}+\sum_{n=n_t}^{+\infty}\frac{\sqrt{n}+1}{e^{c_tn}}<+\infty.
\end{align*}
By applying \cite[Mass distribution principle~4.2]{Falconer2003FractalGeometry}, one obtains that for any $t\in(0,1)$, $\dim F_{f,t}\geq t$.
By Proposition~\ref{prop: F_t is a subset of E_f} and monotonicity of dimension, one obtains that $\dim E_f\geq 1$. 
\end{proof}

\bibliographystyle{siam} 
\bibliography{name}

\end{document}

%% file: gauss_luroth_maps.tex
\begin{figure}[ht]
\centering
\begin{minipage}{.49\textwidth}
\centering
\begin{tikzpicture}
\begin{axis}[
    axis lines = left,
    axis x line=center,
    axis y line=center,
    width=1.1\textwidth,
    height=1.1\textwidth,
    xmin=-.075,
    xmax=1.15,
    ymin=-.075,
    ymax=1.15,
    xtick={1/3,1/2,1},
    xticklabels={$1/3$,$1/2$,1},
    ytick={1},
    xlabel = \(x\),
    ylabel = \(1/x\mod1\),
    axis equal,
]
\draw[dotted] (  0,1) -- (  1,1); 
\draw[dotted] (  1,0) -- (  1,1); 
\foreach \d in {1, 2, ..., 30} {
    \pgfmathtruncatemacro{\Nsamples}{max(8,ceil(300/\d))}
    \addplot[smooth, domain={1/(\d+1):1/\d}, samples=\Nsamples] {1/x-\d};
}
\draw[dotted] (1/2,0) -- (1/2,1); 
\draw[dotted] (1/3,0) -- (1/3,1); 
\draw[dotted] (1/4,0) -- (1/4,1); 
\draw[dotted] (1/5,0) -- (1/5,1); 
\draw[dotted] (1/6,0) -- (1/6,1); 
\draw[dotted] (1/7,0) -- (1/7,1); 
\draw[dotted] (1/8,0) -- (1/8,1); 
\draw[dotted] (1/9,0) -- (1/9,1); 
\end{axis}
\end{tikzpicture}
\end{minipage}
\begin{minipage}{.49\textwidth}
\centering
\begin{tikzpicture}
\begin{axis}[
    axis lines = left,
    axis x line=center,
    axis y line=center,
    width=1.1\textwidth,
    height=1.1\textwidth,
    xmin=-.075,
    xmax=1.15,
    ymin=-.075,
    ymax=1.15,
    xtick={1,1/2,1/3},
    xticklabels={1,$1/2$,$1/3$},
    ytick={1},
    xlabel = \(x\),
    ylabel = \(T(x)\),
    axis equal,
]
\draw[dotted] (  0,1) -- (  1,1);
\foreach \d in {2, 3, ..., 31} {
    \addplot[domain={1/\d:1/(\d-1)}, samples=2] {(\d-1)*(\d*x-1)};
}
\draw[dotted] (1/2,0) -- (1/2,1); 
\draw[dotted] (1/3,0) -- (1/3,1); 
\draw[dotted] (1/4,0) -- (1/4,1); 
\draw[dotted] (1/5,0) -- (1/5,1); 
\draw[dotted] (1/6,0) -- (1/6,1); 
\draw[dotted] (1/7,0) -- (1/7,1); 
\draw[dotted] (1/8,0) -- (1/8,1); 
\draw[dotted] (1/9,0) -- (1/9,1); 
\draw[dotted] (  1,0) -- (  1,1); 
\end{axis}
\end{tikzpicture}
\end{minipage}
\caption{Gauss map and L\"uroth map}
\label{fig: both maps}
\end{figure}
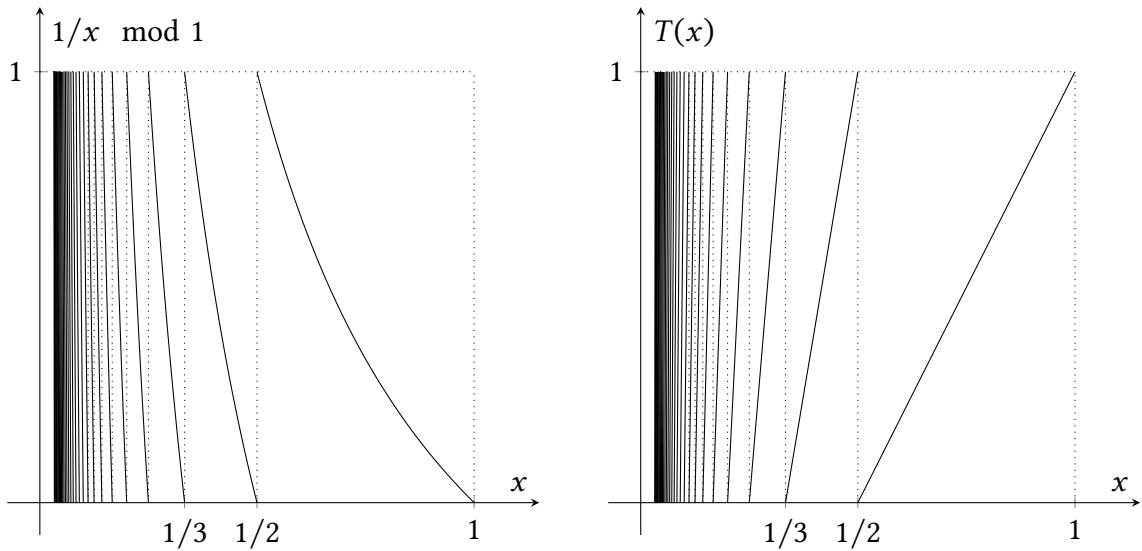